\documentclass[onefignum,onetabnum]{siamonline250211}

\usepackage{amssymb}
\usepackage{mathrsfs}
\usepackage{algorithmic}
\usepackage{placeins}
\usepackage{needspace}

\DeclareMathOperator{\diag}{diag}

\newsiamremark{remark}{Remark}
\crefname{remark}{remark}{remarks}
\Crefname{remark}{Remark}{Remarks}
\AddToHook{env/remark/begin}{\crefalias{theorem}{remark}}

\usepackage[caption=false]{subfig}

\ifpdf
  \DeclareGraphicsExtensions{.pdf,.png,.jpg,.eps}
\else
  \DeclareGraphicsExtensions{.eps}
\fi

\headers{Directional Confocal NLOS Reconstruction}{L. Qiu, Z. Shi, J. Wang, S. Wu}

\title{Reconstruction and Range Characterization for a Directional Confocal Non-Line-of-Sight Imaging Model\thanks{\funding{This work was supported by the National Natural Science Foundation of China (NSFC) under Grant No. 12471399, 92370125.}}}

\author{Lingyun Qiu\thanks{\raggedright Yau Mathematical Sciences Center, Tsinghua University, Beijing 100084, China, and Yanqi Lake Beijing Institute of Mathematical Sciences and Applications, Beijing 101408, China
  (\email{lyqiu@tsinghua.edu.cn}, \email{zqshi@tsinghua.edu.cn}).}
\and Zuoqiang Shi\footnotemark[2]
\and Jianyu Wang\thanks{\raggedright Yau Mathematical Sciences Center, Tsinghua University, Beijing 100084, China
  (\mbox{\email{jy-wang19@mails.tsinghua.edu.cn}}).\protect\endgraf}
\and Shiwei Wu\thanks{\raggedright Qiuzhen College, Tsinghua University, Beijing 100084, China
  (\email{wsw23@mails.tsinghua.edu.cn}).}}

\ifpdf
\hypersetup{
  pdftitle={Reconstruction and Range Characterization for a Directional Confocal Non-Line-of-Sight Imaging Model},
  pdfauthor={Lingyun Qiu, Zuoqiang Shi, Jianyu Wang, and Shiwei Wu}
}
\fi

\begin{document}

\maketitle

\begin{abstract}
We study the reconstruction of a directional albedo field from confocal non-line-of-sight measurements.
For mirror-symmetric vector fields in Sobolev spaces, radial preprocessing reduces the data to spherical 
means of the divergence with centers on the relay wall. Full data determine this divergence uniquely, with
divergence-free fields forming the entire ambiguity. An explicit
Fourier--sine formula recovers the irrotational Helmholtz component and
reproduces the data. The Fourier--sine transforms of the model data form a weighted Hilbert space that we characterize exactly, with norm equal to the Sobolev norm of the reconstructed field. This weighted range remains well defined even when the preprocessed data fail to belong to standard Sobolev spaces. For Schwartz fields in the model class, such Sobolev regularity holds exactly when the divergence has zero depth integral.
Measurements on an open subset of the relay wall, for all radii
$0<r<R_{\max}$, uniquely determine the divergence in the union
of the corresponding balls. An FFT-based algorithm with Stolt interpolation implements the
reconstruction in $O(N^3\log N)$ operations on an $N^3$ grid.
Tests on explicitly defined fields assess reconstruction accuracy and invariance
under divergence-free perturbations. We present imaging examples of potential and vector-field reconstruction
from model-generated, rendered, and measured transients.
\end{abstract}



\section{Introduction}
\label{sec:intro}

Non-line-of-sight (NLOS) imaging reconstructs hidden scenes from light
that reaches the detector along indirect paths. It has potential
applications in robotic vision, autonomous navigation, and remote
sensing \cite{Faccio2020Review}. Time-resolved measurements have enabled
three-dimensional reconstruction of hidden objects
\cite{Velten2012TimeOfFlight} and localization and tracking of moving
objects outside the direct line of sight \cite{Gariepy2016Tracking}.
In the active arrangement shown in \cref{fig:setup}, ultrashort laser
pulses illuminate a visible relay wall, and a single-photon detector
records the returning transient signal
\cite{Kirmani2009TransientImaging,Buttafava2015SPAD}.
The wall thereby acts as a virtual source--detector array.

In confocal acquisition, the virtual source and detector coincide,
so points with the same travel time lie on a sphere centered on the
relay wall. The light-cone transform \cite{OToole2018LCT} and
frequency--wavenumber migration \cite{Lindell2019FK} exploit this
geometry to obtain efficient reconstruction algorithms.
Other computational approaches include fast backprojection
\cite{Arellano2017FastBackProjection} and phasor-field methods
\cite{Liu2019PhasorField,Liu2020PhasorDiffraction}.

\begin{figure}[t!]
  \centering
  \subfloat[Acquisition setup]{
    \label{fig:setup}
    \includegraphics[width=0.4\textwidth]{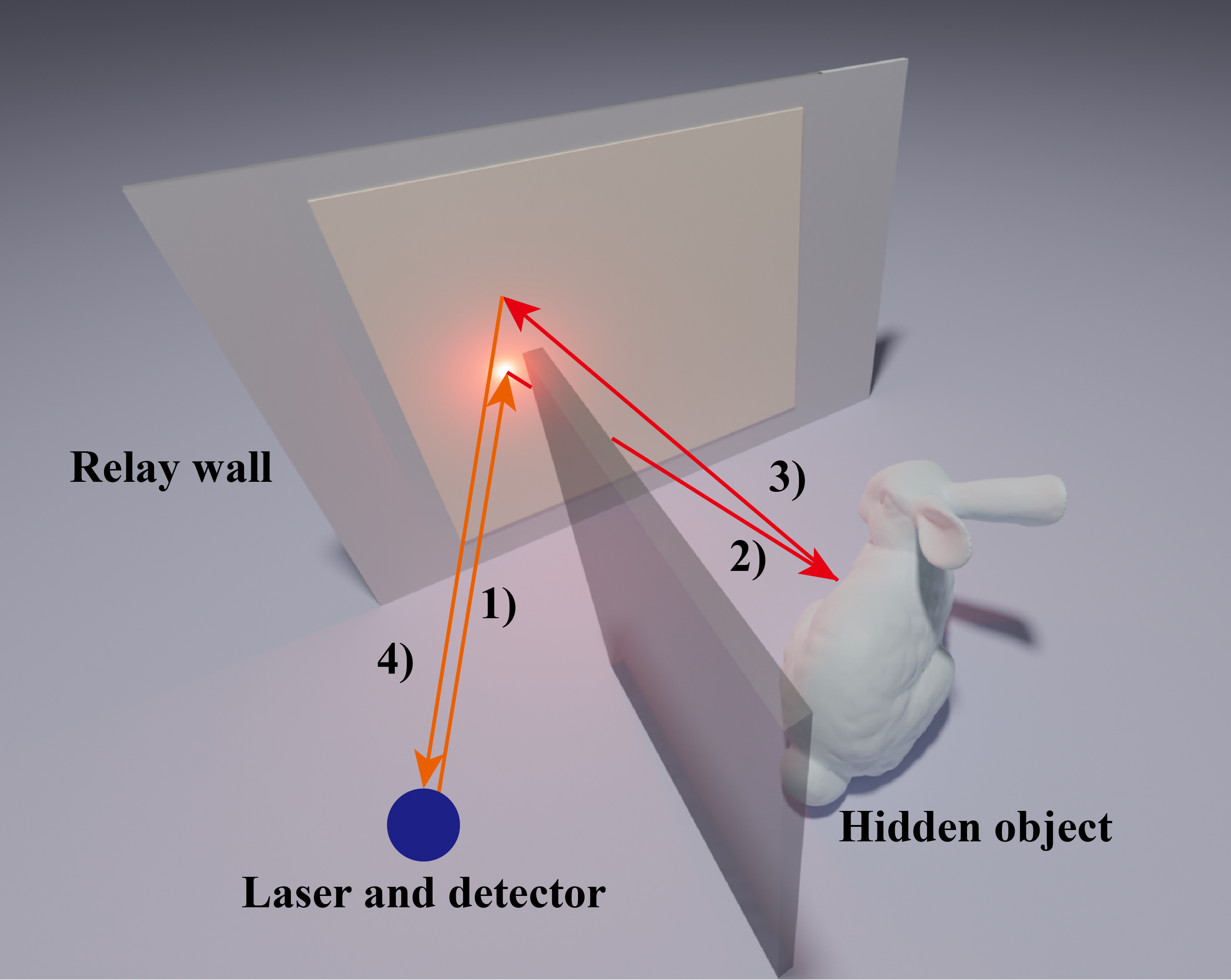}
  }
  \hfill
  \subfloat[Confocal setup]{
    \label{fig:setup_confocal}
    \includegraphics[width=0.25\textwidth]{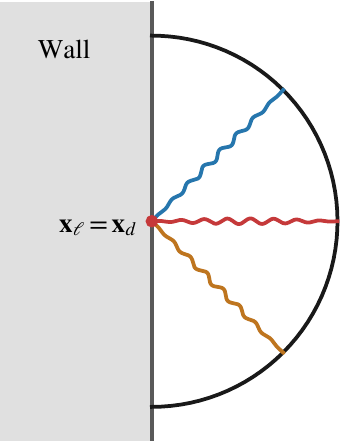}
  }
  \hfill
  \subfloat[Nonconfocal setup]{
    \label{fig:setup_nonconfocal}
    \includegraphics[width=0.25\textwidth]{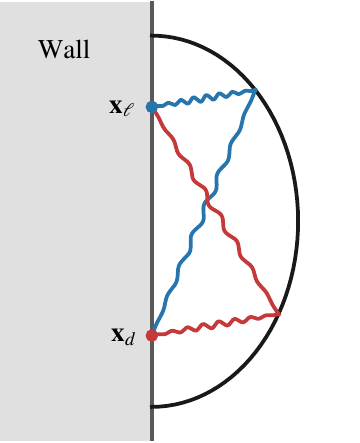}
  }

  \caption{Confocal and nonconfocal NLOS acquisition. The numbered segments in
  panel (a) connect the laser to the wall (1), the wall to the hidden object
  (2), the hidden object back to the wall (3), and the wall to the detector (4).
  In panel (b), the virtual source
  $\mathbf x_\ell$ and detector $\mathbf x_d$ coincide, so a time-of-flight
  level set is a wall-centered sphere. In panel (c), the two wall points differ
  and are the foci of an ellipsoidal level set.}
  \label{fig:setup_all}
  \vspace{-6pt} 
\end{figure}

In scalar models of confocal NLOS imaging, suitable normalization
relates the measurements to spherical means with centers on the relay
wall. A source supported on one side of the wall can be extended evenly
across it, allowing the problem to be formulated on the full space.
Spherical mean transforms also arise in thermoacoustic and photoacoustic
tomography \cite{KuchmentKunyansky2008TAT}, and their inversion and range
have been studied for several acquisition geometries
\cite{FinchPatchRakesh2004Spheres,Kunyansky2007Explicit,
XuWang2005BackProjection,AgranovskyKuchmentQuinto2007Range,
AgranovskyNguyen2010Range}.

For the planar acquisition geometry considered here, classical results
provide inversion formulas for even full-space functions
\cite{Fawcett1985SphericalAverages}, a Fourier description of the range,
and local uniqueness \cite{Andersson1988SphericalAverages}.
An $L^2$ isometry is also available after suitable rescaling and
filtering \cite{KimMoon2024Isometry}.
Even for smooth compactly supported sources, the radius-weighted
spherical means need not be square integrable over all wall centers
and positive radii
\cite{Andersson1988SphericalAverages,Klein2003PhysicallyMeaningful}.
These results provide the scalar theory used below.

Directional models differ from scalar-albedo models because geometric transport 
factors couple intensity with surface orientation. This has motivated computational 
approaches that exploit orientation or surface geometry through partial occlusion, 
surface optimization, and more general illumination and detection patterns
\cite{Heide2019PartialOccluders,Liu2023Arbitrary,Tsai2019SurfaceOptimization}.
The directional light-cone transform is particularly close to the present setting: 
it represents the hidden scene by a three-component directional-albedo field and reconstructs 
it through regularized vector deconvolution \cite{Young2020DLCT}. Related work has jointly 
estimated albedo and surface normals from noisy transients \cite{Liu2021SignalObject} and 
analyzed normal-dependent visibility and blind spots \cite{Liu2019FeatureVisibility}.

We ask what ideal confocal data determine about a volumetric directional field. 
In the model considered here, the data are radially weighted fluxes of a vector field 
$\mathbf u$ through wall-centered spheres. The divergence theorem, followed by 
radial differentiation, reduces them to spherical means of $h=-\nabla\!\cdot\mathbf u$. 
Combining the classical Fourier relation for plane-centered spherical means 
\cite{Andersson1988SphericalAverages,Fawcett1985SphericalAverages} with the Helmholtz 
multiplier yields an explicit Fourier--sine reconstruction of $\mathbb P_\nabla\mathbf u$. 
We prove that divergence-free fields form the entire full-data ambiguity and that this irrotational 
representative has the least Sobolev norm among fields producing the same data. 
This kernel structure is complementary to longitudinal ray measurements, for which potential fields 
are invisible \cite{Sharafutdinov1994TensorFields, Norton1989VectorTomography}.

The same Fourier representation yields an exact energy identity between the reconstructed field 
and the transformed data. This leads to a weighted Hilbert space that we prove to be exactly the 
Fourier--sine transformed data range. Standard Sobolev regularity of the preprocessed data is a 
stronger requirement and can fail even for smooth compactly supported fields, while the corresponding 
weighted range norm remains finite. For Schwartz fields in the model class, this regularity holds exactly 
when the depth integral of $h$ vanishes. Thus depth cancellation governs standard data regularity, 
not the validity of the reconstruction.

For partial data, we use Andersson's local uniqueness theorem
\cite{Andersson1988SphericalAverages} to show that measurements on an open subset of the relay wall, 
for all radii below a fixed bound, determine the divergence in the region covered by the corresponding wall-centered balls.

We implement the reconstruction using fast Fourier transforms and Stolt interpolation. 
The implementation is tested on explicitly defined fields using forward data generated independently 
from the directional surface integral, including a gauge-invariance test under divergence-free 
perturbations. We then present three imaging examples using model-generated, rendered, and measured transients.

The remainder of the paper is organized as follows. \Cref{sec:forward_model} introduces the 
forward model; \cref{sec:reconstruction_stability} gives the full-data reconstruction and 
range characterization; \cref{sec:partial_data_uniqueness} treats partial-data uniqueness; 
and \cref{sec:algorithm_results} presents the numerical implementation and experiments.

\section{Forward model}
\label{sec:forward_model}

Let $\Pi:=\{\mathbf x\in\mathbb R^3:x_3=0\}$ denote the relay wall, and write
$\mathbf x_\Pi=(\mathbf x',0)$ with $\mathbf x'=(x_1,x_2)$. We denote the
wall-centered open ball by $B(\mathbf x',r):=B_{\mathbb R^3}((\mathbf x',0),r)$.
Let $\mathbf x_\ell,\mathbf x_d\in\Pi$ be the virtual source and detector,
and let $t$ denote the hidden-scene round-trip time after subtraction of the
known wall-side travel times. In confocal acquisition,
$\mathbf x_\ell=\mathbf x_d=\mathbf x_\Pi$; with $r=ct/2$, the time-of-flight
level set is the sphere $\partial B(\mathbf x',r)$.

\subsection{Function spaces and Fourier conventions}
\label{subsec:function_spaces}
We write $\mathcal S$ for the Schwartz space and $\mathcal S'$ for the space of
tempered distributions; vector-valued spaces are understood componentwise. Our
Fourier convention is
\begin{equation*}
\begin{aligned}
    \widehat\varphi(\mathbf k)=\mathcal F\varphi(\mathbf k)
    &:={}\int_{\mathbb R^3}\varphi(\mathbf x)
        e^{-i\mathbf k\cdot\mathbf x}\,d\mathbf x,\\
    \mathcal F^{-1}\widehat\varphi(\mathbf x)
    &:={}\frac{1}{(2\pi)^3}\int_{\mathbb R^3}
        \widehat\varphi(\mathbf k)e^{i\mathbf k\cdot\mathbf x}\,d\mathbf k.
\end{aligned}
\end{equation*}
We use the extensions of the Fourier transform to $L^2$ and $\mathcal S'$.
Identities between Fourier representatives are understood almost everywhere.
For vector fields the transform is applied componentwise.

For $s\in\mathbb R$, the space $H^s(\mathbb R^3)$ consists of those elements of
$\mathcal S'(\mathbb R^3)$ for which the scalar norm below is finite. We define
the vector norm componentwise:
\begin{equation*}
\begin{aligned}
    \|\varphi\|_{H^s(\mathbb R^3)}
    &:={}\left(\int_{\mathbb R^3}(1+|\mathbf k|^2)^s
        |\widehat\varphi(\mathbf k)|^2\,d\mathbf k\right)^{1/2},\\
    \|\mathbf v\|_{H^s(\mathbb R^3)}^2
    &:={}\sum_{j=1}^3\|v_j\|_{H^s(\mathbb R^3)}^2
      =\int_{\mathbb R^3}(1+|\mathbf k|^2)^s
        |\widehat{\mathbf v}(\mathbf k)|^2\,d\mathbf k.
\end{aligned}
\end{equation*}

We write
$\mathbf k=(\mathbf k',k_3)$, where $\mathbf k'=(k_1,k_2)$. The partial Fourier transform in
$\mathbf x'=(x_1,x_2)$ is
\begin{equation*}
    \mathcal F_{\mathbf x'}\varphi(\mathbf k',x_3)
    :=
    \int_{\mathbb R^2}
    \varphi(\mathbf x',x_3)e^{-i\mathbf k'\cdot\mathbf x'}\,d\mathbf x'.
\end{equation*}
Subscripts on $\mathcal F$ indicate the transformed variables.

Reflection across the relay wall and its action on vectors are described by
\begin{equation*}
    \sigma(\mathbf x',x_3):=(\mathbf x',-x_3),
    \qquad
    J:=\diag(1,1,-1).
\end{equation*}
For $s\ge0$, set
\begin{equation*}
    H_e^s
    :=
    \left\{
    \mathbf u\in H^s(\mathbb R^3;\mathbb R^3):
    \mathbf u(\sigma\mathbf x)=J\mathbf u(\mathbf x)
    \text{ for a.e. }\mathbf x
    \right\}.
\end{equation*}
Thus $u_1,u_2$ are even and $u_3$ is odd in $x_3$.
For $\mathbf u\in H_e^1$, the divergence belongs to $L^2$ and is even in $x_3$.

For a function $f(\mathbf x',r)$ defined for $r>0$, we write
\begin{equation*}
    f^e(\mathbf x',r):=f(\mathbf x',|r|),
    \qquad
    f^o(\mathbf x',r):=\operatorname{sgn}(r)f(\mathbf x',|r|)
\end{equation*}
for its even and odd extensions in the radial variable.

\subsection{From the transient to directional data}
\label{subsec:confocal_derivation}
Write $\mathbb R^3_+:=\{\mathbf x\in\mathbb R^3:x_3>0\}$.
We adopt the linear directional model of
\cite{Young2020DLCT}. The hidden scene is represented by a compactly
supported directional-albedo field $\mathbf u_+\in C_c^1(\mathbb R^3_+;\mathbb R^3)$,
whose magnitude and direction encode albedo and surface
orientation. In the confocal setting, the transient is
\begin{equation*}
    \tau(\mathbf x_\Pi,\mathbf x_\Pi,t)
    =
    \int_{\mathbb R^3_+}
    \frac{(\mathbf x_\Pi-\mathbf y)\cdot\mathbf u_+(\mathbf y)}
         {|\mathbf x_\Pi-\mathbf y|^5}
    \delta\!\left(2|\mathbf x_\Pi-\mathbf y|-ct\right)
    \,d\mathbf y.
\end{equation*}

Using the reflection defined in \cref{subsec:function_spaces}, extend
$\mathbf u_+$ to $\mathbb R^3$ by
\begin{equation*}
    \mathbf u(\mathbf y):=
    \begin{cases}
        \mathbf u_+(\mathbf y), & y_3>0,\\
        J\mathbf u_+(\sigma\mathbf y), & y_3<0.
    \end{cases}
\end{equation*}
For $y_3>0$,
\begin{equation*}
    (\mathbf x_\Pi-\sigma\mathbf y)\cdot
    \mathbf u(\sigma\mathbf y)
    =
    (\mathbf x_\Pi-\mathbf y)\cdot\mathbf u_+(\mathbf y).
\end{equation*}
Since reflection preserves Lebesgue measure, the half-space integral is
half the corresponding full-space integral. Setting $t=2r/c$, using the
scaling of the delta distribution, and applying the coarea formula then
give
\begin{equation}
\begin{aligned}
    g(\mathbf x',r)
    &:=
    4\tau\!\left(
        \mathbf x_\Pi,\mathbf x_\Pi,\frac{2r}{c}
    \right)\\
    &=
    \int_{\partial B(\mathbf x',r)}
    \frac{(\mathbf x_\Pi-\mathbf y)\cdot\mathbf u(\mathbf y)}
         {r^5}
    \,dS(\mathbf y),
    \qquad
    \mathbf x'\in\mathbb R^2,\quad r>0.
\end{aligned}
\label{eq:directional_forward_surface}
\end{equation}

Since $(\mathbf x_\Pi-\mathbf y)/r$ is the inward unit normal to
$\partial B(\mathbf x',r)$, the data are the negative outward flux of
$\mathbf u$ through the wall-centered sphere, weighted by $r^{-4}$.

\subsection{Scalar reduction, gauge ambiguity}
\label{subsec:divergence_reduction_gauge}
For the extended field $\mathbf u\in C_c^1(\mathbb R^3;\mathbb R^3)$ of
\cref{subsec:confocal_derivation}, the surface formula
\eqref{eq:directional_forward_surface} and the divergence theorem express
the data as an integral over the enclosed ball $B(\mathbf x',r)$:
\begin{equation*}
    g(\mathbf x',r)
    =
    -\frac{1}{r^4}
    \int_{\partial B(\mathbf x',r)}
    \mathbf u(\mathbf y)\cdot\frac{\mathbf y-\mathbf x_\Pi}{r}
    \,dS(\mathbf y)
    =
    -\frac{1}{r^4}
    \int_{B(\mathbf x',r)}\nabla\!\cdot\mathbf u(\mathbf y)\,d\mathbf y.
\end{equation*}
For $\mathbf u\in H_e^1$, define
\begin{equation}
    (\mathcal M\mathbf u)(\mathbf x',r)
    :=
    -\frac{1}{r^4}
    \int_{B(\mathbf x',r)}
    \nabla\!\cdot\mathbf u(\mathbf y)\,d\mathbf y.
    \label{eq:flux_ball_representation}
\end{equation}
The integral is finite for every ball because
$\nabla\!\cdot\mathbf u\in L^2(\mathbb R^3)$. For fields of this class,
the weak Gauss--Green formula identifies this definition with the surface
formula, where $\mathbf u$ on the sphere is its Sobolev trace.
Equation \eqref{eq:flux_ball_representation} depends on $\mathbf u$ only
through its divergence. Hence $\mathcal M(\mathbf u+\mathbf v)=\mathcal M\mathbf u$
for every divergence-free $\mathbf v\in H_e^1$.
For example, if $a\in C_c^\infty(\mathbb R^3)$ is even in $x_3$, then
$\mathbf v:=(\partial_2a,-\partial_1a,0)$ belongs to $H_e^1$ and is divergence free;
choosing $a$ so that $\mathbf v\not\equiv0$ gives a nonzero element of the kernel.
\Cref{cor:full_data_identifiability} shows that divergence-free fields form
the entire full-data ambiguity.

For $h\in L^1_{\mathrm{loc}}(\mathbb R^3)$, define the wall-centered
spherical mean by
\begin{equation*}
    \mathcal Rh(\mathbf x',r)
    :=
    \frac{1}{4\pi r^2}
    \int_{\partial B(\mathbf x',r)}
    h(\mathbf y)\,dS(\mathbf y),
\end{equation*}
for almost every $r>0$.
The following two lemmas give the spherical-mean reduction and the basic
mapping properties of $\mathcal M$.

\begin{lemma}
\label{lem:spherical_mean_reduction}
Let $\mathbf u\in H_e^1$ and let
$g=\mathcal M\mathbf u$, and set
$h:=-\nabla\!\cdot\mathbf u$. Then
\begin{equation}
    \frac{1}{4\pi r^2}
    \frac{\partial}{\partial r}
    \left(r^4g(\mathbf x',r)\right)
    =
    \mathcal Rh(\mathbf x',r)
    \label{eq:radial_derivative_spherical_mean}
\end{equation}
for almost every
$(\mathbf x',r)\in\mathbb R^2\times(0,\infty)$, the radial derivative being understood weakly. 
\end{lemma}

\begin{proof}
Since $h\in L^2\subset L^1_{\mathrm{loc}}$, the coarea formula applied to
\eqref{eq:flux_ball_representation} gives
\begin{equation*}
    \partial_r(r^4g)(\mathbf x',r)
    =\int_{\partial B(\mathbf x',r)}h(\mathbf y)\,dS(\mathbf y)
    =4\pi r^2\mathcal Rh(\mathbf x',r)
\end{equation*}
for almost every $(\mathbf x',r)$. Division by $4\pi r^2$ yields
\eqref{eq:radial_derivative_spherical_mean}.
\end{proof}

\begin{lemma}
\label{lem:forward_operator_bound}
The operator in \eqref{eq:flux_ball_representation} is well defined and
\begin{equation*}
    \mathcal M:
    H_e^1
    \longrightarrow
    L^2_{\mathrm{loc}}(\mathbb R^2\times(0,\infty)).
\end{equation*}
There is an absolute constant $C>0$ such that, for almost every $r>0$,
\begin{equation}
    \|(\mathcal M\mathbf u)(\cdot,r)\|_{L^2(\mathbb R^2)}^2
    \le Cr^{-3}\|\nabla\!\cdot\mathbf u\|_{L^2(\mathbb R^3)}^2.
    \label{eq:forward_slice_bound}
\end{equation}
Consequently, for every $0<a<b<\infty$ there exists $C_{a,b}>0$ such that
\begin{equation}
    \|\mathcal M\mathbf u\|_{L^2(\mathbb R^2\times(a,b))}
    \le
    C_{a,b}\|\nabla\!\cdot\mathbf u\|_{L^2(\mathbb R^3)}
    \le
    C_{a,b}\|\mathbf u\|_{H^1(\mathbb R^3)}.
    \label{eq:forward_local_bound}
\end{equation}
If, for some $d>0$, $\nabla\!\cdot\mathbf u=0$ almost everywhere in the slab
$|x_3|<d$, then $\mathcal M\mathbf u=0$ for $0<r<d$ and
\begin{equation*}
    \|\mathcal M\mathbf u\|_{L^2(\mathbb R^2\times(0,\infty))}
    \le Cd^{-1}\|\nabla\!\cdot\mathbf u\|_{L^2(\mathbb R^3)}.
\end{equation*}
\end{lemma}

The proof is given in \cref{app:forward_fourier_sine_proofs}.

\begin{remark}
\label{rem:finite_resolution_surface}
Let $\Omega_+\Subset\mathbb R^3_+$ be a bounded domain with smooth boundary,
and set $\Omega:=\Omega_+\cup\sigma(\Omega_+)$ and $S:=\partial\Omega$.
Write $\chi_\Omega$ for its indicator, $\mathbf n$ for the outward unit normal,
and $\delta_S$ for surface measure on $S$. In the sense of distributions,
$\nabla\chi_\Omega=-\mathbf n\delta_S$.
For constant directional albedo $a_0>0$, the oriented surface field
$\mathbf u_S:=a_0\mathbf n\delta_S=-a_0\nabla\chi_\Omega$ is therefore a gradient.
Mollifying with a nonnegative radial $\eta\in C_c^\infty(B_{\mathbb R^3}(0,1))$, $\int\eta=1$, $\eta_\varepsilon(\mathbf x):=\varepsilon^{-3}\eta(\mathbf x/\varepsilon)$, gives
\begin{equation*}
    \mathbf u_\varepsilon:=\mathbf u_S*\eta_\varepsilon
    =\nabla\phi_\varepsilon,
    \qquad \phi_\varepsilon:=-a_0\chi_\Omega*\eta_\varepsilon.
\end{equation*}
Since $\phi_\varepsilon$ is even in $x_3$ and belongs to $C_c^\infty$,
$\mathbf u_\varepsilon\in H_e^s$ for every $s\ge0$.
Oriented surface fields of this kind are thus irrotational already at the level
of the model, which is one reason to single out the irrotational component
in what follows.
\end{remark}

\section{Full-data reconstruction and range characterization}
\label{sec:reconstruction_stability}

This section develops the full-data reconstruction and characterizes the
associated transformed data range and Sobolev data regularity. Throughout,
\begin{equation*}
    \mathbf u\in H_e^1,\qquad g=\mathcal M\mathbf u,
    \qquad
    h:=-\nabla\!\cdot\mathbf u,
\end{equation*}
with data on the full observation domain $\mathbb R^2\times(0,\infty)$.

\subsection{Canonical reconstruction}
\label{subsec:reconstruction_map}

We first recover the scalar $h=-\nabla\!\cdot\mathbf u$ from the full data
and then construct the irrotational vector field with this divergence.

Define the preprocessed data by
\begin{equation}
    \gamma(\mathbf x',r)
    :=
    \frac{1}{4\pi r}\partial_r
    \bigl(r^4g(\mathbf x',r)\bigr).
    \label{eq:preprocessing_operator}
\end{equation}
By \cref{lem:spherical_mean_reduction},
$\gamma(\mathbf x',r)=r\mathcal Rh(\mathbf x',r)$.
For model data, this preprocessing is reversible; see
\eqref{eq:preprocessing_inverse}.

Let $\gamma^o$ denote the odd extension of $\gamma$ in the radial variable.
Since $\mathcal Rh(\mathbf x',r)$ is even in $r$,
$\gamma^o(\mathbf x',r)=r\mathcal Rh(\mathbf x',|r|)$.
By \eqref{eq:preprocessed_slice_bound},
$\gamma^o\in\mathcal S'(\mathbb R^3)$.
We define
\begin{equation}
    \Theta_\gamma(\mathbf k',\rho)
    :=
    \frac{i}{2}\widehat{\gamma^o}(\mathbf k',\rho),
    \qquad \rho>0.
    \label{eq:sine_transform_theta}
\end{equation}
Whenever the classical integral is defined,
\begin{equation*}
    \Theta_\gamma(\mathbf k',\rho)
    =
    \int_0^\infty
    \mathcal F_{\mathbf x'}\gamma(\mathbf k',r)
    \sin(\rho r)\,dr.
\end{equation*}

To state the Fourier relation, set
\begin{equation*}
\begin{aligned}
    \Lambda
    :=
    \left\{
        (\mathbf k',\rho)\in\mathbb R^2\times(0,\infty):
        \rho>|\mathbf k'|
    \right\},
    \qquad
    q(\mathbf k',\rho)
    :=
    \sqrt{\rho^2-|\mathbf k'|^2}.
\end{aligned}
\end{equation*}
Thus $\rho^2=|\mathbf k'|^2+q^2$ on $\Lambda$.
The Fourier relation for plane-centered spherical means
\cite{Andersson1988SphericalAverages,Fawcett1985SphericalAverages}
takes the following form in our normalization.

\begin{proposition}
\label{prop:fourier_sine_representation}
Let $\mathbf u\in H_e^1$, set
$g=\mathcal M\mathbf u$ and $h=-\nabla\!\cdot\mathbf u$, and let
$\gamma$ be defined by \eqref{eq:preprocessing_operator}.
Then the Fourier transform of $\gamma^o$ is represented by the locally
integrable function
\begin{equation}
    \widehat{\gamma^o}(\mathbf k',\rho)
    =
    \begin{cases}
    \displaystyle-i\operatorname{sgn}(\rho)
      \frac{\widehat h(\mathbf k',q(\mathbf k',|\rho|))}
           {q(\mathbf k',|\rho|)}, & |\rho|>|\mathbf k'|,\\[6pt]
    0, & |\rho|\le|\mathbf k'|.
    \end{cases}
    \label{eq:full_odd_cone_representation}
\end{equation}
Consequently,
\begin{equation}
    \Theta_\gamma(\mathbf k',\rho)
    =
    \frac{\widehat h(\mathbf k',q(\mathbf k',\rho))}
         {2q(\mathbf k',\rho)}
    \qquad
    \text{for a.e. }(\mathbf k',\rho)\in\Lambda,
    \label{eq:fourier_sine_density}
\end{equation}
and hence
\begin{equation}
    \widehat h(\mathbf k',k_3)
    =
    2|k_3|\Theta_\gamma(\mathbf k',|\mathbf k|)
    \qquad
    \text{for a.e. }\mathbf k\in\mathbb R^3.
    \label{eq:cone_factor_identity}
\end{equation}
\end{proposition}

\begin{proof}
The full-space formula
\eqref{eq:full_odd_cone_representation}
is proved in \cref{app:forward_fourier_sine_proofs}.
Equation \eqref{eq:fourier_sine_density} follows from
\eqref{eq:sine_transform_theta}.
For $k_3\ne0$, $q(\mathbf k',|\mathbf k|)=|k_3|$.
Since $\widehat h$ is even in $k_3$, substituting
$\rho=|\mathbf k|$ into \eqref{eq:fourier_sine_density} gives
\eqref{eq:cone_factor_identity}.
\end{proof}

Proposition~\ref{prop:fourier_sine_representation} gives the scalar
reconstruction directly. Define
\begin{equation}
    \widehat{h_g}(\mathbf k)
    :=
    2|k_3|\Theta_\gamma(\mathbf k',|\mathbf k|).
    \label{eq:scalar_synthesis_spectrum}
\end{equation}
By \eqref{eq:cone_factor_identity}, $h_g=h$ for model data. We then select
the irrotational vector field with this divergence by setting
\begin{equation}
    \widehat{\mathcal Ig}(\mathbf k)
    :=
    \frac{i\mathbf k}{|\mathbf k|^2}\widehat{h_g}(\mathbf k),
    \qquad \mathbf k\ne0.
    \label{eq:vector_synthesis_spectrum}
\end{equation}
Thus $\mathcal Ig$ is the irrotational Fourier solution of
$-\nabla\!\cdot\mathbf v=h_g$.

Let $\mathbb P_\nabla$ denote the Helmholtz projection onto
irrotational fields, defined by
\begin{equation*}
    \widehat{\mathbb P_\nabla\mathbf v}(\mathbf k)
    :=
    \frac{\mathbf k\otimes\mathbf k}{|\mathbf k|^2}
    \widehat{\mathbf v}(\mathbf k),
    \qquad \mathbf k\ne0.
\end{equation*}
The following theorem shows that
$\mathcal Ig=\mathbb P_\nabla\mathbf u$ for exact model data.

\begin{theorem}
\label{thm:canonical_reconstruction}
Let $s\ge1$ and $\mathbf u\in H_e^s$, with
$g=\mathcal M\mathbf u$.
Then
\begin{equation}
    \mathcal Ig
    =
    \mathbb P_\nabla\mathbf u
    \in H_e^s,
    \qquad
    -\nabla\!\cdot\mathcal Ig
    =
    -\nabla\!\cdot\mathbf u.
    \label{eq:helmholtz_projection}
\end{equation}
Moreover, $\mathcal Ig$ reproduces the data:
\begin{equation}
    \mathcal M(\mathcal Ig)=g.
    \label{eq:reconstruction_data_consistency}
\end{equation}
\end{theorem}

\begin{proof}
For model data, \eqref{eq:cone_factor_identity} gives
$\widehat{h_g}=\widehat h$. Since
$\widehat h(\mathbf k)
=-i\mathbf k\cdot\widehat{\mathbf u}(\mathbf k)$,
\eqref{eq:vector_synthesis_spectrum} yields
\begin{equation*}
    \widehat{\mathcal Ig}(\mathbf k)
    =
    \frac{\mathbf k\otimes\mathbf k}{|\mathbf k|^2}
    \widehat{\mathbf u}(\mathbf k).
\end{equation*}
Hence $\mathcal Ig=\mathbb P_\nabla\mathbf u$.

The multiplier $\mathbf k\otimes\mathbf k/|\mathbf k|^2$ is an orthogonal
projection, so
$|\widehat{\mathcal Ig}|\le|\widehat{\mathbf u}|$ and hence
$\|\mathcal Ig\|_{H^s}\le\|\mathbf u\|_{H^s}$.
Moreover,
$\mathbf P_\nabla(J\mathbf k)=J\mathbf P_\nabla(\mathbf k)J$,
so the reflection symmetry is preserved.
Thus $\mathcal Ig\in H_e^s$.

Finally,
\begin{equation*}
    \widehat{-\nabla\!\cdot\mathcal Ig}
    =
    -i\mathbf k\cdot
    \frac{i\mathbf k}{|\mathbf k|^2}\widehat h
    =
    \widehat h.
\end{equation*}
Thus $-\nabla\!\cdot\mathcal Ig=h$.
\eqref{eq:flux_ball_representation} then gives
$\mathcal M(\mathcal Ig)=\mathcal M\mathbf u=g$.
\end{proof}

The canonical reconstruction also determines the full-data ambiguity.
\begin{corollary}
\label{cor:full_data_identifiability}
For $\mathbf u_1,\mathbf u_2\in H_e^1$,
\begin{equation}
    \mathcal M\mathbf u_1=\mathcal M\mathbf u_2
    \quad\Longleftrightarrow\quad
    \nabla\!\cdot(\mathbf u_1-\mathbf u_2)=0.
    \label{eq:full_data_ambiguity}
\end{equation}
\end{corollary}

\begin{proof}
If $\mathcal M\mathbf u_1=\mathcal M\mathbf u_2=:g$, then
\cref{thm:canonical_reconstruction} gives
\begin{equation*}
    -\nabla\!\cdot\mathbf u_1
    =
    -\nabla\!\cdot\mathcal Ig
    =
    -\nabla\!\cdot\mathbf u_2.
\end{equation*}
The converse follows directly from
\eqref{eq:flux_ball_representation}.
\end{proof}

Hence the fields producing a fixed datum $g$ differ from $\mathcal Ig$ by a divergence-free field. The next corollary identifies $\mathcal Ig$ as the minimum-norm representative of this class.
\begin{corollary}
\label{cor:minimum_norm_reconstruction}
Let $s\ge1$ and $g\in\mathcal M(H_e^s)$. Then $\mathcal Ig$ is the unique
minimizer of the $H^s$ norm among all fields with divergence $-h_g$, and
also among all fields in $H_e^s$ producing the data $g$. Equivalently,
\begin{equation}
    \|\mathcal Ig\|_{H^s}
    =
    \min_{\substack{
        \mathbf v\in H^s(\mathbb R^3;\mathbb R^3)\\
        -\nabla\!\cdot\mathbf v=h_g}}
    \|\mathbf v\|_{H^s}
    =
    \min_{\substack{
        \mathbf v\in H_e^s\\
        \mathcal M\mathbf v=g}}
    \|\mathbf v\|_{H^s}.
    \label{eq:minimum_norm_reconstruction}
\end{equation}
\end{corollary}

\begin{proof}
By \cref{thm:canonical_reconstruction},
$\mathcal Ig\in H_e^s$ satisfies
$-\nabla\!\cdot\mathcal Ig=h_g$ and
$\mathcal M(\mathcal Ig)=g$.
Thus it belongs to both feasible sets in
\eqref{eq:minimum_norm_reconstruction}.

Let $\mathbf v\in H^s(\mathbb R^3;\mathbb R^3)$ satisfy
$-\nabla\!\cdot\mathbf v=h_g$, and set
$\mathbf w=\mathbf v-\mathcal Ig$.
Then $\nabla\!\cdot\mathbf w=0$, so
$\mathbf k\cdot\widehat{\mathbf w}(\mathbf k)=0$ a.e.
By \eqref{eq:vector_synthesis_spectrum},
$\widehat{\mathcal Ig}$ is parallel to $\mathbf k$.
Hence $\mathbf w$ is $H^s$-orthogonal to $\mathcal Ig$, and
\begin{equation*}
    \|\mathbf v\|_{H^s}^2
    =
    \|\mathcal Ig\|_{H^s}^2
    +
    \|\mathbf v-\mathcal Ig\|_{H^s}^2.
\end{equation*}
Thus $\mathcal Ig$ is the unique minimizer under the divergence constraint.

For $\mathbf v\in H_e^s$,
\cref{cor:full_data_identifiability} and
\eqref{eq:reconstruction_data_consistency} give
\begin{equation*}
    \mathcal M\mathbf v=g
    \quad\Longleftrightarrow\quad
    \nabla\!\cdot(\mathbf v-\mathcal Ig)=0
    \quad\Longleftrightarrow\quad
    -\nabla\!\cdot\mathbf v=h_g.
\end{equation*}
Hence the two feasible sets have the same minimum-norm representative.
\end{proof}

\begin{remark}
\label{rem:wave_spectral_support}
The Fourier--sine reconstruction has a wave-equation interpretation.
For $h\in\mathcal S(\mathbb R^3)$ even in $x_3$, let $w$ solve
\begin{equation*}
    w_{rr}=\Delta_{\mathbf x}w,\qquad
    w(\mathbf x,0)=0,\qquad w_r(\mathbf x,0)=h(\mathbf x).
\end{equation*}
Kirchhoff's formula \cite{John1981PlaneWaves} gives
\begin{equation*}
    w(\mathbf x,r)
    =\frac{1}{4\pi r}\int_{\partial B_{\mathbb R^3}(\mathbf x,r)}h(\mathbf y)\,dS(\mathbf y)
    \qquad(r>0).
\end{equation*}
At the wall, this is $w((\mathbf x',0),r)=r\mathcal Rh(\mathbf x',r)
=\gamma(\mathbf x',r)$. Since $w$ is odd in $r$, its wall trace for all
real $r$ is $\gamma^o$.

The spatial Fourier transform of the solution is
\begin{equation*}
    \widehat w(\mathbf k,r)
    =\frac{\sin(r|\mathbf k|)}{|\mathbf k|}\widehat h(\mathbf k).
\end{equation*}
Thus the radial frequencies satisfy $\rho^2=|\mathbf k'|^2+k_3^2$.
Taking the wall trace retains $(\mathbf k',\rho)$ and integrates over
$k_3$, yielding the cone support $|\rho|\ge|\mathbf k'|$. On $\Lambda$,
the Stolt change recovers the normal wavenumber magnitude $|k_3|=q$.
The derivative $w_r$ has initial displacement $h$ and zero initial velocity,
as in planar thermoacoustic tomography~\cite{XuFengWang2002Planar}.
\end{remark}

\subsection{Isometry and the exact transformed range}
\label{subsec:weighted_energy_range}

We first derive an energy identity for the reconstruction. It determines
a natural norm on the transformed data, which is then shown to describe
the exact transformed range.

\begin{proposition}
\label{prop:weighted_energy}
Let $s\ge0$ and let $\mathbf u\in H_e^1$, with
$g=\mathcal M\mathbf u$ and $h=-\nabla\!\cdot\mathbf u$.
Then, allowing either side to take the value $+\infty$,
\begin{align}
    \|\mathcal Ig\|_{H^s}^2
    &=
    8\int_\Lambda
    (1+\rho^2)^s\frac{q}{\rho}
    |\Theta_\gamma(\mathbf k',\rho)|^2
    \,d\mathbf k'\,d\rho
    \label{eq:weighted_exact_energy}\\
    &=
    \int_{\mathbb R^3}
    \frac{(1+|\mathbf k|^2)^s}{|\mathbf k|^2}
    |\widehat h(\mathbf k)|^2\,d\mathbf k.
    \label{eq:field_spectral_norm}
\end{align}
The same identities hold for differences of two exact data sets.
\end{proposition}

\begin{proof}
By \cref{thm:canonical_reconstruction} with $s=1$,
$\mathcal Ig$ is well defined. From
\eqref{eq:vector_synthesis_spectrum},
\begin{equation*}
    |\widehat{\mathcal Ig}(\mathbf k)|^2
    =
    \frac{4k_3^2}{|\mathbf k|^2}
    |\Theta_\gamma(\mathbf k',|\mathbf k|)|^2.
\end{equation*}
On $k_3>0$, use the change of variables $\rho=(|\mathbf k'|^2+k_3^2)^{1/2}$. Then $q(\mathbf k',\rho)=k_3$ and $dk_3=\rho q(\mathbf k',\rho)^{-1}d\rho$, with $\rho>|\mathbf k'|$.
The contribution from $k_3<0$ is
identical, and therefore
\begin{align*}
    \|\mathcal Ig\|_{H^s}^2
    &=
    2\int_{\mathbb R^2}
    \int_{|\mathbf k'|}^\infty
    (1+\rho^2)^s
    \frac{4q^2}{\rho^2}
    |\Theta_\gamma(\mathbf k',\rho)|^2
    \frac{\rho}{q}\,d\rho\,d\mathbf k'\\
    &=
    8\int_\Lambda
    (1+\rho^2)^s\frac q\rho
    |\Theta_\gamma(\mathbf k',\rho)|^2
    \,d\mathbf k'\,d\rho.
\end{align*}
which proves \eqref{eq:weighted_exact_energy}.

Using \eqref{eq:cone_factor_identity} in
\eqref{eq:vector_synthesis_spectrum} gives
\begin{equation*}
    |\widehat{\mathcal Ig}(\mathbf k)|^2
    =
    \frac{|\widehat h(\mathbf k)|^2}{|\mathbf k|^2},
\end{equation*}
and hence \eqref{eq:field_spectral_norm}.
The difference identities follow by linearity.
\end{proof}

Motivated by \eqref{eq:weighted_exact_energy}, let $Y^s$ be the space of
almost-everywhere equivalence classes of measurable functions
$F:\Lambda\to\mathbb C$ satisfying
$F(-\mathbf k',\rho)=\overline{F(\mathbf k',\rho)}$ and
\begin{equation}
    \|F\|_{Y^s}^2
    :=
    8\int_\Lambda
    (1+\rho^2)^s\frac q\rho
    |F(\mathbf k',\rho)|^2\,d\mathbf k'\,d\rho
    <\infty.
    \label{eq:weighted_data_norm}
\end{equation}
With the real part of the corresponding weighted $L^2$ inner product,
$Y^s$ is a real Hilbert space.

For $s\ge1$, set
\begin{equation*}
    X^s:=\mathcal M(H_e^s),
    \qquad
    \|g\|_{X^s}:=\|\mathcal Ig\|_{H^s}.
\end{equation*}
By \cref{thm:canonical_reconstruction,prop:weighted_energy},
$g\mapsto\Theta_\gamma$ is a linear isometric embedding of $X^s$ into
$Y^s$. The following theorem shows that it is onto.

\begin{theorem}
\label{thm:exact_transformed_range}
Let $s\ge1$. The map $g\mapsto\Theta_\gamma$ is a linear isometric
isomorphism from $X^s$ onto $Y^s$.
For $F\in Y^s$, define $\mathbf v_F$ by
\begin{equation}
    \widehat{\mathbf v_F}(\mathbf k)
    :=
    \frac{2i\mathbf k|k_3|}{|\mathbf k|^2}
    F(\mathbf k',|\mathbf k|),
    \qquad k_3\ne0.
    \label{eq:range_synthesis_spectrum}
\end{equation}
Then $\mathbf v_F\in H_e^s$ and the inverse map is
$F\mapsto\mathcal M\mathbf v_F$.
In particular, $X^s$ is a real Hilbert space and
\begin{equation*}
    X^s
    =
    \left\{
        g\in\mathcal M(H_e^1):
        \mathcal Ig\in H^s
    \right\}.
\end{equation*}
\end{theorem}

\begin{proof}
By \cref{thm:canonical_reconstruction,prop:weighted_energy},
the map
\begin{equation*}
    g\longmapsto\Theta_\gamma
\end{equation*}
is linear and isometric from $X^s$ into $Y^s$.
It remains to prove surjectivity.

Let $F\in Y^s$, and define $\mathbf v_F$ by
\eqref{eq:range_synthesis_spectrum}.
Using the same change of variables as in
\cref{prop:weighted_energy}, we obtain
\begin{equation*}
    \|\mathbf v_F\|_{H^s}^2
    =
    8\int_\Lambda
    (1+\rho^2)^s\frac{q}{\rho}
    |F(\mathbf k',\rho)|^2
    \,d\mathbf k'\,d\rho
    =
    \|F\|_{Y^s}^2.
\end{equation*}
Hence $\mathbf v_F\in H^s(\mathbb R^3;\mathbb R^3)$.

The Hermitian condition on $F$ gives
$\widehat{\mathbf v_F}(-\mathbf k)
=\overline{\widehat{\mathbf v_F}(\mathbf k)}$, while
\eqref{eq:range_synthesis_spectrum} gives
$\widehat{\mathbf v_F}(J\mathbf k)
=J\widehat{\mathbf v_F}(\mathbf k)$.
Therefore $\mathbf v_F$ is real valued and belongs to $H_e^s$.

Since $s\ge1$, the function
$h_F:=-\nabla\!\cdot\mathbf v_F$ belongs to $L^2(\mathbb R^3)$.
Its Fourier transform is
\begin{equation*}
    \widehat h_F(\mathbf k)
    =
    2|k_3|F(\mathbf k',|\mathbf k|).
\end{equation*}
Set $g_F:=\mathcal M\mathbf v_F$.
For $(\mathbf k',\rho)\in\Lambda$, let
$q=q(\mathbf k',\rho)$. Since
$\sqrt{|\mathbf k'|^2+q^2}=\rho$,
we have
$\widehat h_F(\mathbf k',q)=2qF(\mathbf k',\rho)$.
Applying \cref{prop:fourier_sine_representation} to $g_F$ therefore gives
\begin{equation*}
    \Theta_{\gamma_F}(\mathbf k',\rho)
    =
    \frac{\widehat h_F(\mathbf k',q)}
         {2q}
    =
    F(\mathbf k',\rho)
\end{equation*}
for almost every $(\mathbf k',\rho)\in\Lambda$.
Thus $g_F\mapsto F$, and the map
$g\mapsto\Theta_\gamma$ is onto $Y^s$.
Since it is also isometric, it is injective, and its inverse is
$F\mapsto\mathcal M\mathbf v_F$.

Because $Y^s$ is complete, the isometric isomorphism implies that
$X^s$ is a real Hilbert space.

Finally, let $g\in\mathcal M(H_e^1)$.
If $\mathcal Ig\in H^s$, then
\cref{thm:canonical_reconstruction} with $s=1$ gives
$\mathcal Ig\in H_e^1$.
Together with $\mathcal Ig\in H^s$, this implies
$\mathcal Ig\in H_e^s$.
By data consistency,
\begin{equation*}
    g=\mathcal M(\mathcal Ig)\in\mathcal M(H_e^s)=X^s.
\end{equation*}
Conversely, if $g\in X^s$, then
$g=\mathcal M\mathbf u$ for some $\mathbf u\in H_e^s$, and
\cref{thm:canonical_reconstruction} gives
$\mathcal Ig=\mathbb P_\nabla\mathbf u\in H_e^s$.
Hence $\mathcal Ig\in H^s$, proving
\begin{equation*}
    X^s
    =
    \left\{
        g\in\mathcal M(H_e^1):
        \mathcal Ig\in H^s
    \right\}.
\end{equation*}
\end{proof}

\subsection{Standard Sobolev data norms and depth cancellation}
\label{subsec:data_norms_cancellation}

The norm in $Y^s$ is adapted to reconstruction. The standard Sobolev norm
of the preprocessed data carries an additional singular weight near the
cone boundary. We first identify this weight and then examine how depth
cancellation and angular separation affect it.

\begin{proposition}
\label{prop:standard_data_comparison}
Let $s\ge0$ and $\mathbf u\in H_e^1$, with
$g=\mathcal M\mathbf u$ and $h=-\nabla\!\cdot\mathbf u$.
Then, in the extended sense,
\begin{equation}
    \|\gamma^o\|_{H^s}^2
    =
    \int_{\mathbb R^3}
    \frac{(1+|\mathbf k'|^2+|\mathbf k|^2)^s}
         {|k_3||\mathbf k|}
    |\widehat h(\mathbf k)|^2\,d\mathbf k.
    \label{eq:data_spectral_norm}
\end{equation}
Hence $\gamma^o\in H^s$ if and only if the integral is finite, and in
that case
\[
    \|\mathcal Ig\|_{H^s}\le\|\gamma^o\|_{H^s}.
\]
\end{proposition}

\begin{proof}
By \cref{prop:fourier_sine_representation} and
$\widehat{\gamma^o}=-2i\Theta_\gamma$ for $\rho>0$,
\begin{equation*}
    \|\gamma^o\|_{H^s}^2
    =
    8\int_\Lambda
    (1+|\mathbf k'|^2+\rho^2)^s
    |\Theta_\gamma(\mathbf k',\rho)|^2
    \,d\mathbf k'\,d\rho.
\end{equation*}
Substituting \eqref{eq:fourier_sine_density} and, for fixed
$\mathbf k'$, changing variables from $\rho$ to
$k_3=q(\mathbf k',\rho)>0$, with
$d\rho=(k_3/\rho)\,dk_3$, gives
\eqref{eq:data_spectral_norm} after using the evenness of
$\widehat h$ in $k_3$.
Here $\rho=|\mathbf k|$, so
$1+|\mathbf k'|^2+\rho^2 =1+|\mathbf k'|^2+|\mathbf k|^2$.
Relative to \eqref{eq:weighted_exact_energy}, the weights in the two
integrals over $\Lambda$ have ratio
\begin{equation*}
    \left(\frac{1+|\mathbf k'|^2+\rho^2}{1+\rho^2}\right)^s
    \frac{\rho}{q}.
\end{equation*}
Both factors are at least one, proving the inequality.
The polynomial factor lies between $1$ and $2^s$, whereas $\rho/q$ is unbounded as $q/\rho\to0$.
\end{proof}

For Schwartz fields, standard data regularity is characterized
by a depth-cancellation condition.

\begin{proposition}
\label{prop:sobolev_cancellation_characterization}
Let $\mathbf u\in H_e^1\cap\mathcal S(\mathbb R^3;\mathbb R^3)$,
set $h=-\nabla\!\cdot\mathbf u$ and $g=\mathcal M\mathbf u$, and let
$\gamma$ be given by \eqref{eq:preprocessing_operator}. Define
$\bar h(\mathbf x'):=\int_{\mathbb R}h(\mathbf x',z)\,dz$.
For every $s\in\mathbb R$,
\begin{equation}
    \gamma^o\in H^s(\mathbb R^3)
    \quad\Longleftrightarrow\quad \bar h=0.
    \label{eq:cancellation_equivalence}
\end{equation}
Moreover, for each fixed $\mathbf k'\ne\mathbf0$,
\begin{equation}
    \Theta_\gamma(\mathbf k',\rho)
    =\frac{\widehat h(\mathbf k',0)}{2q}+O(q)
    \qquad\text{as }\rho\downarrow|\mathbf k'|.
    \label{eq:cone_edge_asymptotic}
\end{equation}
\end{proposition}

\begin{proof}
Integrating in depth gives
\begin{equation}
    \widehat h(\mathbf k',0)=\mathcal F_{\mathbf x'}\bar h(\mathbf k').
    \label{eq:depth_cancellation_condition}
\end{equation}
Smoothness and evenness give
$\widehat h(\mathbf k',\kappa)=\widehat h(\mathbf k',0)+O(\kappa^2)$;
substituting $\kappa=q$ in \eqref{eq:fourier_sine_density} proves
\eqref{eq:cone_edge_asymptotic}.

If $\bar h\ne0$, \eqref{eq:depth_cancellation_condition} and continuity give
a bounded open set $V\subset\mathbb R^2$ away from the origin and
$c,\varepsilon>0$ such that
$|\widehat h(\mathbf k',\kappa)|\ge c$ on $V\times(0,\varepsilon)$.
Since $V$ is bounded, choose $M>0$ such that
$\sqrt{|\mathbf k'|^2+\varepsilon^2}\le M$ for $\mathbf k'\in V$, we obtain
\begin{equation*}
    \int_V
      \int_{|\mathbf k'|}^{\sqrt{|\mathbf k'|^2+\varepsilon^2}}
      |\Theta_\gamma(\mathbf k',\rho)|^2\,d\rho\,d\mathbf k'
    \ge\frac{c^2}{4M}
      \int_V\int_0^\varepsilon\frac{d\kappa}{\kappa}\,d\mathbf k'
      =\infty,
\end{equation*}
using $d\rho=(\kappa/\rho)\,d\kappa$.
Hence $\widehat{\gamma^o}$ is not square integrable on this bounded frequency set. Since the Sobolev weight is bounded below there, $\gamma^o\notin H^s$ for every $s\in\mathbb R$.

Conversely, if $\bar h=0$, evenness, Taylor's formula, and Schwartz bounds give
\begin{equation*}
    \widehat h(\mathbf k',k_3)=k_3^2a(\mathbf k',k_3),
    \qquad
    |a(\mathbf k',k_3)|\le C_N(1+|\mathbf k'|)^{-N}
    \quad (|k_3|\le1),
\end{equation*}
where $a$ is smooth and $N$ is arbitrary.
For $s\ge0$ and $|k_3|\le1$, the integrand in
\eqref{eq:data_spectral_norm} satisfies
\begin{equation*}
    \frac{(1+|\mathbf k'|^2+|\mathbf k|^2)^s}
         {|k_3||\mathbf k|}|\widehat h|^2
    \le C_{s,N}(1+|\mathbf k'|^2)^s
         (1+|\mathbf k'|)^{-2N}k_3^2,
\end{equation*}
because $|k_3|^3/|\mathbf k|\le k_3^2$.
The right-hand side is integrable over $\mathbb R^2\times[-1,1]$ for $N>s+1$.
On $|k_3|\ge1$, the factor $1/(|k_3||\mathbf k|)$ is bounded, and rapid
decay of $\widehat h$ makes the remaining integral finite.
By \cref{prop:standard_data_comparison},
$\gamma^o\in H^s$ for every $s\ge0$, hence also for $s<0$.
This completes \eqref{eq:cancellation_equivalence}.
\end{proof}

Depth cancellation characterizes finiteness of the standard Sobolev norm, but not its uniform comparability with the reconstruction norm. Such a comparison is recovered after excluding a neighborhood of the cone boundary.

\begin{corollary}
\label{cor:cone_interior_stability}
Let $s\ge1$ and $\mathbf u\in H_e^s$, with $g=\mathcal M\mathbf u$ and
$\gamma$ given by \eqref{eq:preprocessing_operator}.
For $0<\delta<1$, define
\begin{equation*}
    \Lambda_\delta
    :=
    \left\{
        (\mathbf k',\rho)\in\Lambda:
        q(\mathbf k',\rho)\ge\delta\rho
    \right\},
    \qquad
    \widehat{\mathcal I_\delta g}(\mathbf k)
    :=
    \mathbf1_{\{|k_3|\ge\delta|\mathbf k|\}}
    \widehat{\mathcal Ig}(\mathbf k).
\end{equation*}
With
\begin{equation}
    N_{\delta,s}(\gamma)^2
    :=8\int_{\Lambda_\delta}(1+|\mathbf k'|^2+\rho^2)^s
      |\Theta_\gamma(\mathbf k',\rho)|^2\,d\mathbf k'\,d\rho,
    \label{eq:cone_interior_data_norm}
\end{equation}
one has $N_{\delta,s}(\gamma)<\infty$ and
\begin{equation}
    \frac{\delta}{(2-\delta^2)^s}N_{\delta,s}(\gamma)^2
    \le\|\mathcal I_\delta g\|_{H^s}^2
    \le N_{\delta,s}(\gamma)^2.
    \label{eq:cone_interior_stability}
\end{equation}
The lower constant is optimal for each fixed $\delta$ and $s$.
The same inequalities hold for differences of exact data in $\mathcal M(H_e^s)$.
\end{corollary}

\begin{proof}
By \cref{thm:canonical_reconstruction}, $\mathcal Ig\in H^s$, so its
Fourier truncation also belongs to $H^s$.
Since $q(\mathbf k',|\mathbf k|)=|k_3|$, the indicators
$\mathbf1_{\Lambda_\delta}(\mathbf k',|\mathbf k|)$ and
$\mathbf1_{\{|k_3|\ge\delta|\mathbf k|\}}(\mathbf k)$ agree almost everywhere.
The same change of variables as in \cref{prop:weighted_energy} gives
\begin{equation*}
    \|\mathcal I_\delta g\|_{H^s}^2
    =8\int_{\Lambda_\delta}(1+\rho^2)^s\frac q\rho
      |\Theta_\gamma(\mathbf k',\rho)|^2\,d\mathbf k'\,d\rho.
\end{equation*}
On $\Lambda_\delta$, $\delta\le q/\rho\le1$ and
$1\le(1+|\mathbf k'|^2+\rho^2)/(1+\rho^2)\le2-\delta^2$.
Comparing the weights proves \eqref{eq:cone_interior_stability} and
finiteness of $N_{\delta,s}(\gamma)$.

To prove optimality, choose nonzero real-valued
$F_n\in C_c^\infty(\Lambda)$, even in $\mathbf k'$, with
\begin{equation*}
    \operatorname{supp}F_n\subset
    \left\{(\mathbf k',\rho)\in\Lambda:
      n<\rho<n+1,\quad \delta<\frac q\rho<\delta+\frac1n\right\}
\end{equation*}
for all sufficiently large $n$. By \cref{thm:exact_transformed_range},
there exist $g_n\in X^s$ with $\Theta_{\gamma_n}=F_n$.
Since
$|\mathbf k'|^2/\rho^2=1-(q/\rho)^2$,
the ratio
\begin{equation*}
    \frac{(1+\rho^2)^s q}{\rho(1+|\mathbf k'|^2+\rho^2)^s}
\end{equation*}
converges uniformly on $\operatorname{supp}F_n$ to
$\delta/(2-\delta^2)^s$. The ratio of the two energy integrals is a
weighted average of this ratio, so
\begin{equation*}
    \lim_{n\to\infty}
    \frac{\|\mathcal I_\delta g_n\|_{H^s}^2}
         {N_{\delta,s}(\gamma_n)^2}
    =\frac{\delta}{(2-\delta^2)^s}.
\end{equation*}
This proves optimality.
\end{proof}

Both regularity regimes occur within the model class. The test fields in \cref{subsec:analytic_validation} provide cancelling and noncancelling examples. For the mollified surface fields of \cref{rem:finite_resolution_surface}, the compactly supported function $\bar\phi_\varepsilon$ is nonzero and hence cannot be harmonic; therefore $\bar h=-\Delta'\bar\phi_\varepsilon\not\equiv0$. They nevertheless have transformed data in every $Y^s$ and are reconstructed exactly.

Moreover, \cref{app:near_tangential_family} shows that even within the depth-cancelling Schwartz class there is no uniform reverse estimate of the standard Sobolev data norm by the reconstruction norm. Thus depth cancellation characterizes finiteness of the standard data norm, but not its uniform equivalence with the reconstruction norm.

\section{Partial-data uniqueness}
\label{sec:partial_data_uniqueness}

The scalar reduction also gives a local uniqueness result for the divergence from partial data. For an open measurement region $\Gamma\subset\mathbb R^2$ and
$R_{\max}>0$, define
\begin{equation*}
    \Omega_{\Gamma,R_{\max}}
    :=
    \bigcup_{\mathbf x'\in\Gamma}
    B(\mathbf x',R_{\max}).
\end{equation*}

\begin{theorem}
\label{thm:local_uniqueness}
Let $\Gamma\subset\mathbb R^2$ be open, let $R_{\max}>0$, and let
$\mathbf u_1,\mathbf u_2\in H_e^1$. Then
\begin{equation*}
    \begin{aligned}
    &\mathcal M\mathbf u_1=\mathcal M\mathbf u_2
    \quad\text{a.e.\ on }\Gamma\times(0,R_{\max})\\
    &\quad\Longleftrightarrow\quad
    \nabla\!\cdot(\mathbf u_1-\mathbf u_2)=0
    \quad\text{a.e.\ on }\Omega_{\Gamma,R_{\max}}.
    \end{aligned}
\end{equation*}
\end{theorem}

\begin{proof}
Let $\mathbf w:=\mathbf u_1-\mathbf u_2$ and
$\delta h:=-\nabla\!\cdot\mathbf w$.
Then $\delta h\in L^2(\mathbb R^3)$ and is even in $x_3$.
Suppose first that the data agree. Fix $\mathbf x_0'\in\Gamma$ and
$0<R<R_{\max}$, and choose $\varepsilon>0$ such that
\[
    B_{\mathbb R^2}(\mathbf x_0',\varepsilon)\subset\Gamma.
\]
On
$B_{\mathbb R^2}(\mathbf x_0',\varepsilon)\times(0,R)$
we have $\mathcal M\mathbf w=0$. Hence
\cref{lem:spherical_mean_reduction} gives
\begin{equation*}
    \mathcal R\delta h=0
    \qquad\text{a.e. on }
    B_{\mathbb R^2}(\mathbf x_0',\varepsilon)\times(0,R).
\end{equation*}
By \cref{lem:radial_lift}, the corresponding spherical-mean transform
vanishes on
\[
    B_{\mathbb R^2}(\mathbf x_0',\varepsilon)
    \times B_{\mathbb R^3}(0,R).
\]
The local uniqueness theorem
\cite{Andersson1988SphericalAverages}
therefore gives
\[
    \delta h=0
    \qquad\text{a.e.\ in }B(\mathbf x_0',R).
\]
These balls cover $\Omega_{\Gamma,R_{\max}}$; a countable subcover gives
$\delta h=0$ almost everywhere on that region.

Conversely, if $\delta h=0$ almost everywhere on
$\Omega_{\Gamma,R_{\max}}$, every wall-centered ball $B(\mathbf x',r)$ with $\mathbf x'\in\Gamma$ and radius
less than $R_{\max}$ lies in that region. The ball formula
\eqref{eq:flux_ball_representation} then gives $\mathcal M\mathbf w=0$ on
the observation set.
\end{proof}

Taking $\Gamma=\mathbb R^2$ and $R_{\max}=m$, and then letting
$m\to\infty$, recovers the full-data ambiguity characterized in
\cref{cor:full_data_identifiability}.

\section{Algorithm and experiments}
\label{sec:algorithm_results}

We first describe the discrete Fourier reconstruction. Tests on analytic
fields assess its accuracy and invariance under divergence-free
perturbations, using independently generated data. Three imaging examples
then examine the reconstructed potential and field for model-generated,
rendered, and measured transients.

\subsection{Algorithm}
\label{subsec:proposed_algorithm}
The input consists of samples of $g(\mathbf x',r)$. We form $\gamma$ by
\eqref{eq:preprocessing_operator} and take the three-dimensional
FFT of its odd radial extension. Internal zero padding refines the radial
frequency grid used for Stolt interpolation. Denote the resulting discrete approximation of $\widehat{\gamma^o}$ by $G_d(\mathbf k',\rho_\ell)$.
At each admissible frequency $\rho_\ell>|\mathbf k'|$, set
\begin{equation}
 q_\ell=\sqrt{\rho_\ell^2-|\mathbf k'|^2},\qquad
 \widehat h_{\rm rec}(\mathbf k',k_3)
 =\operatorname{interp}_{q_\ell\to|k_3|}
 \bigl[iq_\ell G_d(\mathbf k',\rho_\ell)\bigr].
 \label{eq:discrete_scalar_reconstruction}
\end{equation}
Here $\widehat h_{\rm rec}$ is the discrete counterpart of
$\widehat{h_g}$ in \eqref{eq:scalar_synthesis_spectrum}.
The cone factor is applied before interpolation, as prescribed by
\eqref{eq:cone_factor_identity}. We use a local even quadratic fit near $q=0$, corresponding to $k_3=0$ after Stolt interpolation. The potential and field follow from
\begin{equation}
 \widehat\phi_{\rm rec}(\mathbf k)=\frac{\widehat h_{\rm rec}(\mathbf k)}{|\mathbf k|^2},
 \qquad \widehat{\mathbf u}_{\rm rec}(\mathbf k)=i\mathbf k\widehat\phi_{\rm rec}(\mathbf k)
 \quad(\mathbf k\ne0),
 \label{eq:discrete_field_reconstruction}
\end{equation}
with the frequency origin set to zero, followed by inverse FFTs.
The two multipliers implement the field synthesis in
\eqref{eq:vector_synthesis_spectrum}.
The FFTs dominate the $O(N^3\log N)$ time and
$O(N^3)$ memory costs on an $N^3$ grid \cite{Lindell2019FK}.

\subsection{Validation on test fields}
\label{subsec:analytic_validation}
We test two smooth compactly supported gradient fields, with and without
depth cancellation, using the same sampling and reconstruction parameters.
Let
\begin{equation*}
 \begin{aligned}
 \eta(s)&=
 \begin{cases}
 \exp\!\left(1-(1-s)^{-1}-8s\right),&0\le s<1,\\
 0,&s\ge1,
 \end{cases}\\[3pt]
 s_j(\mathbf x)&=\frac{(x_1-x_0)^2+(x_2-y_0)^2}{a_j^2}
                  +\frac{(x_3-z_j)^2}{b_j^2},
 \end{aligned}
\end{equation*}
where $(x_0,y_0)=(0.04,-0.03)\,\mathrm m$ and $j\in\{\mathrm c,\mathrm{nc}\}$.
Define the positive-depth potentials by
\begin{equation*}
 \phi_{\mathrm c}^+(\mathbf x)=\frac{x_3-z_{\mathrm c}}{b_{\mathrm c}}
      \eta(s_{\mathrm c}(\mathbf x)),
 \qquad
 \phi_{\mathrm{nc}}^+(\mathbf x)=\eta(s_{\mathrm{nc}}(\mathbf x)),
\end{equation*}
with $(a_{\mathrm c},b_{\mathrm c},z_{\mathrm c})=(0.24,0.24,0.36)\,\mathrm m$
and $(a_{\mathrm{nc}},b_{\mathrm{nc}},z_{\mathrm{nc}})=(0.36,0.12,0.18)\,\mathrm m$.
For each case, set
$\phi_j(\mathbf x',x_3)=\phi_j^+(\mathbf x',x_3)+\phi_j^+(\mathbf x',-x_3)$
and $\mathbf u_j=\nabla\phi_j$.
The sum is even in $x_3$, hence $\mathbf u_j\in H_e^1$.
The positive-depth supports are a ball and an oblate ellipsoid, respectively,
both separated from the wall. Only the first field satisfies the
depth-cancellation condition of \cref{prop:sobolev_cancellation_characterization}.

We compute the data directly from the directional surface integral
\eqref{eq:directional_forward_surface}, independently of the divergence
reduction, using $256$ Gauss--Legendre nodes per integration coordinate.
The $12\,\mathrm m\times12\,\mathrm m$ wall is sampled at $512^2$ points,
with $512$ radial bins at $\Delta r=0.02\,\mathrm m$.
Both reconstructions use twofold internal radial zero padding and a
four-point even quadratic cone-edge fit. No field smoothing or support
clipping is applied, and the normalization is fixed by the Fourier
convention without fitting a scale to the reference field.

The relative vector errors
$\|\mathbf u_{\rm rec}-\mathbf u\|_{\ell^2}/\|\mathbf u\|_{\ell^2}$,
evaluated over all three components on the entire returned $512^3$ grid,
are $1.89\%$ and $6.23\%$ for the cancelling and noncancelling cases,
respectively. Since the fields differ in geometry and depth, these errors
do not isolate the effect of depth cancellation.

\begin{figure}[!htbp]
 \centering
 \includegraphics[width=0.94\textwidth]{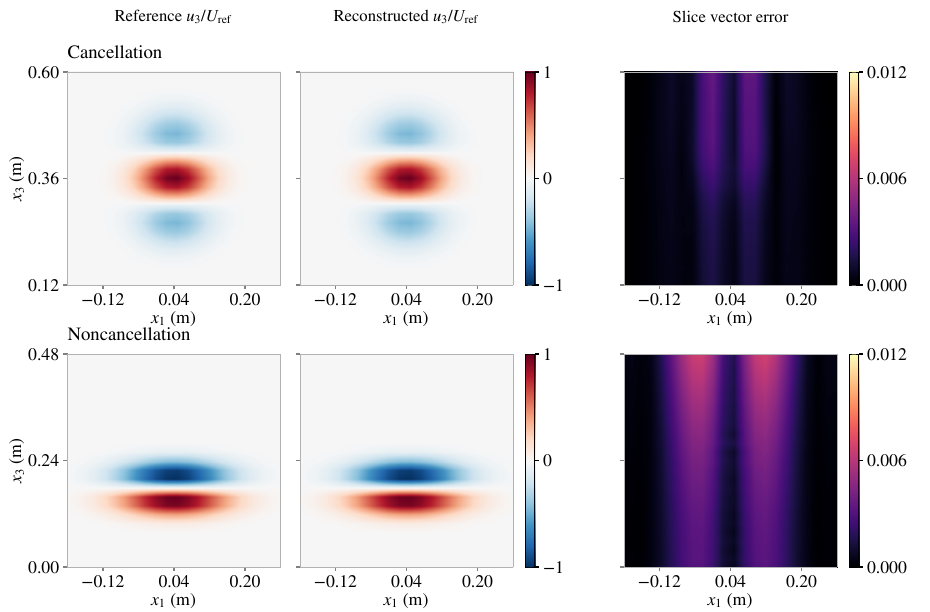}
 \caption{Reconstruction of smooth compactly supported test fields with depth cancellation (top) and
 without it (bottom), on the central slice $x_2=-0.03\,\mathrm m$.
 Columns show the reference $u_3/U_{\rm ref}$, reconstructed
 $u_{{\rm rec},3}/U_{\rm ref}$, and the vector error
 $|\mathbf u_{\rm rec}-\mathbf u|/U_{\rm ref}$ on the same slice, where
 $U_{\rm ref}$ is the maximum reference vector magnitude over the full grid
 for each case. The component and error color scales are shared across
 rows. }
 \label{fig:analytic_reconstruction}
\end{figure}

\Cref{fig:analytic_reconstruction} shows the distinct depth parity of the
axial component: it is even about the target center in the cancelling case
and odd in the noncancelling case.

For the gauge test, use the cancelling potential $\phi_{\mathrm c}$ to define
$\mathbf v=(\partial_2\phi_{\mathrm c},-\partial_1\phi_{\mathrm c},0)$,
which belongs to $H_e^1$ and is divergence free. We separately evaluate the
surface data of $\mathbf u_{\mathrm c}$ and $\mathbf u_{\mathrm c}+\mathbf v$
on a $128^2$ wall grid covering $3\,\mathrm m\times3\,\mathrm m$, with $128$
radial bins at $\Delta r=0.02\,\mathrm m$, using $192$ Gauss--Legendre nodes per integration coordinate and the same
zero-padding and cone-edge interpolation settings as above.
The relative discrete $\ell^2$ differences are $4.17\times10^{-8}$ for the
data and $2.88\times10^{-6}$ for the full reconstructed vector field,
each measured against its unperturbed array.

\subsection{Three imaging examples}
\label{subsec:imaging_examples}
We apply the reconstruction to a model-generated Letter-T transient,
rendered Stanford Bunny transients, and measured Stanford statue data.
The rendered and measured data undergo background correction. The radial derivative in
\eqref{eq:preprocessing_operator} is evaluated using
Savitzky--Golay differentiation, and Gaussian smoothing is applied during
Fourier synthesis. These operations are used to stabilize the reconstruction for
finite-window and noisy data and are not part of the continuous inversion
formula.
The rendered and measured transients include transport effects not represented
by \eqref{eq:directional_forward_surface}; the displayed vector components
are therefore not interpreted as surface-normal estimates.

\paragraph{Model-generated data}
The Letter-T example uses a model-generated transient and the corresponding
reference potential on a $64\times64$ wall grid covering
$0.6\,\mathrm m\times0.6\,\mathrm m$.
The data contain $512$ radial bins with $\Delta r=0.00125\,\mathrm m$,
and the target is centered at a depth of approximately $0.32\,\mathrm m$.
In \cref{fig:letter_t_reconstruction}, the reconstructed potential projection reproduces
the horizontal bar and vertical stem of the reference, with slight blurring
at the edges.

\paragraph{Rendered data}
The Stanford Bunny transients are from the synthetic Zaragoza NLOS data set,
rendered as described in \cite{jarabo2014}. We retain the first $512$ time bins
and use $256\times256$ wall samples over a $2\,\mathrm m\times2\,\mathrm m$
region, with $\Delta r=0.003\,\mathrm m$; the object is approximately
$0.5\,\mathrm m$ from the wall.
In \cref{fig:bunny_reconstruction}, the potential and field-magnitude images
retain the ears and body outline, with interruptions near the head and
an uneven response across the body.

\begin{figure}[!tbp]
 \setlength{\abovecaptionskip}{4pt}
 \centering
  \includegraphics[width=0.74\textwidth]{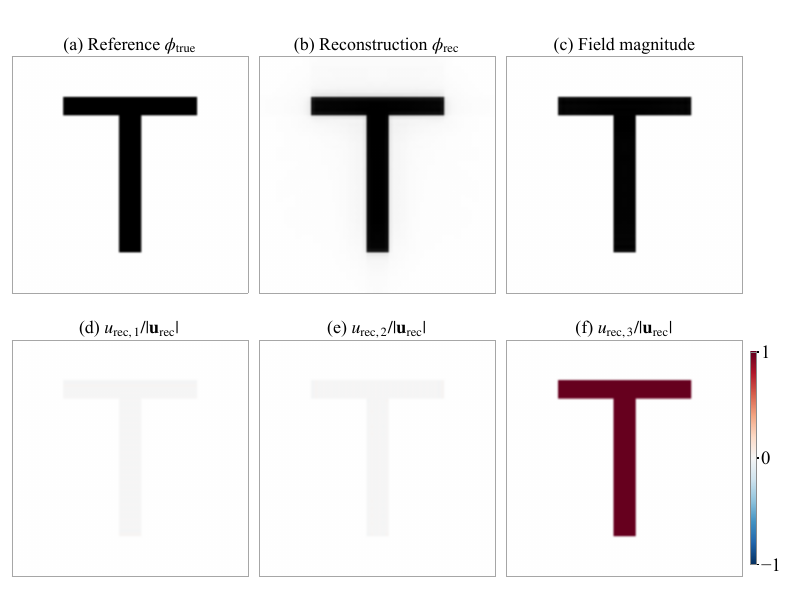}
 \caption{Letter-T: (a) reference potential, (b) reconstructed potential,
 (c) field magnitude, and (d)--(f) normalized field components.
 Panel (b) projects $\max(\phi_{\rm rec},0)$ by taking the depthwise maximum.
 Panels (c)--(f) select the depth maximizing $|\mathbf u_{\rm rec}|$ with
 $u_{{\rm rec},3}>0$. These conventions also apply to
 \cref{fig:bunny_reconstruction,fig:statue_reconstruction}.}
  \label{fig:letter_t_reconstruction}
\end{figure}

\paragraph{Measured data}
The $180$-min Stanford statue acquisition \cite{Lindell2019FK} uses
$512$ retained, aligned time bins and $512\times512$ wall measurements
over $2\,\mathrm m\times2\,\mathrm m$. The $32$ ps bin spacing gives
$\Delta r=0.0048\,\mathrm m$.
In \cref{fig:statue_reconstruction}, the extended arm, bent torso, and base are visible.
The field-magnitude image is more localized than
the potential, which has a diffuse halo; fine surface detail remains unresolved.

\begin{figure}[!tbp]
 \setlength{\abovecaptionskip}{4pt}
 \centering
 \includegraphics[width=0.74\textwidth]{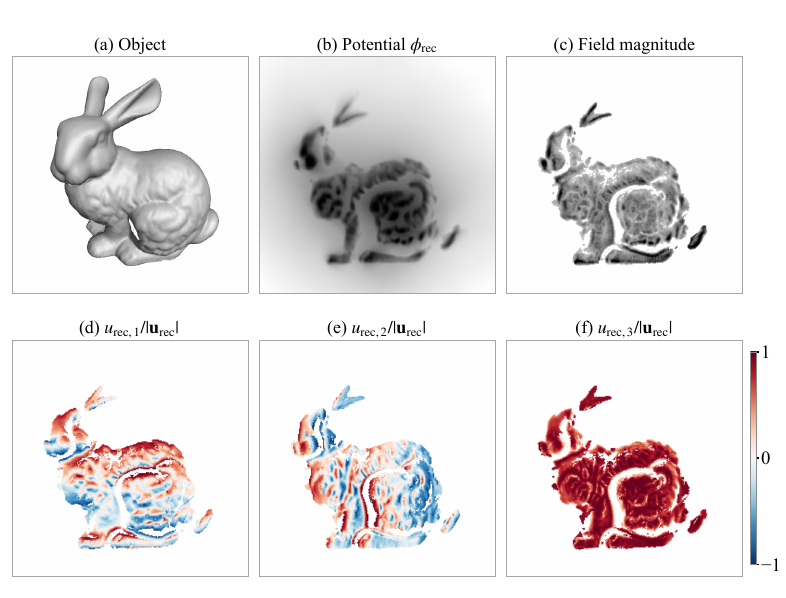}
 \caption{Bunny: (a) reference scene, (b) reconstructed potential,
 (c) field magnitude, and (d)--(f) normalized field components,
 using the display conventions in \cref{fig:letter_t_reconstruction}.}
 \label{fig:bunny_reconstruction}
\end{figure}

\section{Conclusion}
\label{sec:conclusion}

The directional confocal model considered here reduces the recovery of a vector field to a scalar spherical-mean problem for its divergence. This identifies the irrotational Helmholtz component as the natural data-consistent reconstruction and accounts for the full-data gauge ambiguity. The transformed range further shows that exact reconstruction does not require standard Sobolev regularity of the preprocessed data; for Schwartz fields, depth cancellation characterizes precisely when that additional regularity is present. Partial measurements retain the same scalar structure and determine the divergence in the corresponding observable region.

The numerical experiments verify the Fourier implementation on test fields and illustrate the reconstruction on model-generated, rendered, and measured data. The present analysis does not quantify the effects of finite apertures, radial truncation, or raw-data noise. These questions, together with reconstruction for nonconfocal acquisition, provide natural directions for further work.

\begin{figure}[!tbp]
 \centering
 \includegraphics[width=0.74\textwidth]{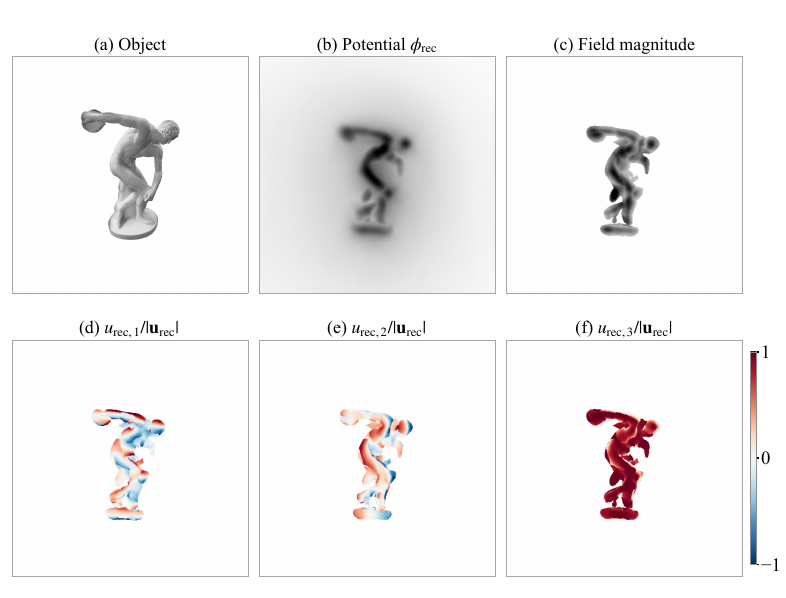}
 \caption{Statue, $180$ min: (a) reference scene, (b) reconstructed potential,
 (c) field magnitude, and (d)--(f) normalized field components,
 using the display conventions in \cref{fig:letter_t_reconstruction}.}
 \label{fig:statue_reconstruction}
\end{figure}

\appendix

\section{Proofs for the forward operator and Fourier--sine representation}
\label{app:forward_fourier_sine_proofs}

\begin{proof}[Proof of \cref{lem:forward_operator_bound}]
Let $\mathbf u\in H_e^1$ and $h:=-\nabla\!\cdot\mathbf u$.
The Cauchy--Schwarz inequality gives
\begin{equation*}
    |(\mathcal M\mathbf u)(\mathbf x',r)|^2
    \le
    C r^{-5}
    \int_{B(\mathbf x',r)}|h(\mathbf y)|^2\,d\mathbf y.
\end{equation*}
For fixed
$\mathbf y=(\mathbf y',y_3)$, the set of wall centers
$\mathbf x'$ for which $\mathbf y\in B(\mathbf x',r)$ has area at most
$\pi r^2$. Fubini's theorem gives \eqref{eq:forward_slice_bound}; integrating in
$r$ yields
\begin{equation*}
    \int_a^b\int_{\mathbb R^2}
    |(\mathcal M\mathbf u)(\mathbf x',r)|^2\,d\mathbf x'\,dr
    \le C\int_a^b r^{-3}\,dr\,\|h\|_{L^2(\mathbb R^3)}^2,
\end{equation*}
which proves \eqref{eq:forward_local_bound}. If $r<d$, every ball centered on the
relay plane lies in the slab $|x_3|<d$, so the data vanish by the ball formula.
Integrating \eqref{eq:forward_slice_bound} over $(d,\infty)$ and using
$\int_d^\infty r^{-3}\,dr=(2d^2)^{-1}$ gives the global bound.
\end{proof}

We also record the inverse of the preprocessing map used in
\cref{subsec:reconstruction_map}.
For $\mathbf u\in H_e^1$, set $g=\mathcal M\mathbf u$ and
$h=-\nabla\!\cdot\mathbf u$. The estimate
$|r^4g(\mathbf x',r)|\le Cr^{3/2}\|h\|_2$ from
\eqref{eq:flux_ball_representation} and the identity
$\partial_r(r^4g)=4\pi r\gamma$ give
\begin{equation}
    \lim_{r\downarrow0}r^4g(\mathbf x',r)=0,
    \qquad
    g(\mathbf x',r)=\frac{4\pi}{r^4}\int_0^r t\gamma(\mathbf x',t)\,dt.
    \label{eq:preprocessing_inverse}
\end{equation}

\begin{proof}[Proof of \cref{prop:fourier_sine_representation}]
We prove \eqref{eq:full_odd_cone_representation} for real
$h\in L^2(\mathbb R^3)$ that are even in $x_3$, with
$\gamma=r\mathcal Rh$ as in \cref{lem:spherical_mean_reduction}.
Parameterizing the sphere by
$\mathbf y=(\mathbf x',0)+r\boldsymbol\omega$ and applying Jensen's inequality
gives, for almost every $r>0$,
\begin{align}
    \|\gamma(\cdot,r)\|_{L^2(\mathbb R^2)}^2
    &\le
    \frac{r^2}{4\pi}
    \int_{\mathbb R^2}\int_{\mathbb S^2}
    |h((\mathbf x',0)+r\boldsymbol\omega)|^2
    \,d\boldsymbol\omega\,d\mathbf x' \notag\\
    &=
    \frac r2
    \int_{-r}^{r}\int_{\mathbb R^2}
    |h(\mathbf y',y_3)|^2\,d\mathbf y'\,dy_3
    \le
    \frac r2\|h\|_{L^2(\mathbb R^3)}^2.
    \label{eq:preprocessed_slice_bound}
\end{align}
Integrating in $r$ gives, for every $R>0$,
\begin{equation*}
    \|\mathbf1_{\{|r|<R\}}\gamma^o\|_{L^2(\mathbb R^2\times\mathbb R)}^2
    \le\frac{R^2}{2}\|h\|_{L^2(\mathbb R^3)}^2.
\end{equation*}
Thus $\gamma^o$ has polynomial $L^2$ growth in the radial variable and
defines a tempered distribution. The same estimate shows that
$h\mapsto\gamma^o$ is continuous from $L^2(\mathbb R^3)$ to
$L^2(\mathbb R^2\times(-R,R))$ for every $R>0$.

For even $h\in\mathcal S(\mathbb R^3)$, Fourier inversion and
\begin{equation*}
    \frac{1}{4\pi}\int_{\mathbb S^2}
    e^{ir\mathbf k\cdot\boldsymbol\omega}\,d\boldsymbol\omega
    =\frac{\sin(r|\mathbf k|)}{r|\mathbf k|}
\end{equation*}
give, after Fourier transformation in $\mathbf x'$,
\begin{equation}
    \mathcal F_{\mathbf x'}\gamma^o(\mathbf k',r)
    =
    \frac{1}{2\pi}\int_{\mathbb R}
    \widehat h(\mathbf k',k_3)
    \frac{\sin(r|\mathbf k|)}{|\mathbf k|}\,dk_3.
    \label{eq:spherical_mean_partial_fourier}
\end{equation}
The formula holds for all real $r$ by oddness. For
$\psi\in\mathcal S(\mathbb R^2\times\mathbb R)$, with
$\langle\cdot,\cdot\rangle$ denoting bilinear duality, the identity
\begin{equation*}
    \mathcal F_r\bigl(\sin(r\lambda)\bigr)(\rho)
    =\frac{\pi}{i}
      \bigl(\delta(\rho-\lambda)-\delta(\rho+\lambda)\bigr),
    \qquad \lambda>0,
\end{equation*}
therefore yields
\begin{equation}
    \langle\widehat{\gamma^o},\psi\rangle
    =
    \frac{1}{2i}\int_{\mathbb R^3}
    \frac{\widehat h(\mathbf k)}{|\mathbf k|}
    \bigl(\psi(\mathbf k',|\mathbf k|)-\psi(\mathbf k',-|\mathbf k|)\bigr)
    \,d\mathbf k.
    \label{eq:odd_spectrum_pairing}
\end{equation}
The interchanges are justified by Schwartz decay and local integrability
of $|\mathbf k|^{-1}$.

To pass to even $h\in L^2$, choose real-valued, $x_3$-even
$h_j\in\mathcal S(\mathbb R^3)$ with
$h_j\to h$ in $L^2$, and set $\gamma_j=r\mathcal Rh_j$.
The function
\begin{equation*}
    m_\psi(\mathbf k)
    :=
    \frac{\psi(\mathbf k',|\mathbf k|)-\psi(\mathbf k',-|\mathbf k|)}{|\mathbf k|}
\end{equation*}
is bounded near $\mathbf k=0$ by the mean value formula and rapidly
decreasing at infinity. In particular, $m_\psi\in L^2(\mathbb R^3)$.
Plancherel and Cauchy--Schwarz imply convergence of the right-hand side of
\eqref{eq:odd_spectrum_pairing}. Estimate \eqref{eq:preprocessed_slice_bound}
and partial Plancherel also give
\begin{equation*}
    \left|\left\langle
      \widehat{\gamma_j^o}-\widehat{\gamma^o},\psi
    \right\rangle\right|
    \le C\|h_j-h\|_{L^2}
    \int_{\mathbb R}|r|^{1/2}
      \|\mathcal F_\rho\psi(\cdot,r)\|_{L^2(\mathbb R^2)}\,dr
    \longrightarrow0.
\end{equation*}
The integral is finite because $\mathcal F_\rho\psi$ is Schwartz.
It follows that \eqref{eq:odd_spectrum_pairing} holds for the original
$h\in L^2$ and every $\psi\in\mathcal S(\mathbb R^3)$.

Using the evenness of $\widehat h$, restrict to $k_3>0$ and set
$\rho=(|\mathbf k'|^2+k_3^2)^{1/2}$.
Then $q(\mathbf k',\rho)=k_3$ and
$d\rho=(q/\rho)\,dk_3$.
Equation~\eqref{eq:odd_spectrum_pairing} becomes
\begin{equation}
    \langle\widehat{\gamma^o},\psi\rangle
    =
    \frac{1}{i}\int_{\mathbb R^2}\int_{|\mathbf k'|}^\infty
    \frac{\widehat h(\mathbf k',q)}{q}
    \bigl(\psi(\mathbf k',\rho)-\psi(\mathbf k',-\rho)\bigr)
    \,d\rho\,d\mathbf k'.
    \label{eq:cone_density_pairing}
\end{equation}
Define
\begin{equation*}
    G_h(\mathbf k',\rho)
    :=
    \begin{cases}
    \displaystyle
    -i\operatorname{sgn}(\rho)
    \frac{\widehat h(\mathbf k',q(\mathbf k',|\rho|))}
         {q(\mathbf k',|\rho|)},
    & |\rho|>|\mathbf k'|,\\[6pt]
    0,
    & |\rho|\le|\mathbf k'|.
    \end{cases}
\end{equation*}
We show that $G_h$ is locally integrable and represents $\widehat{\gamma^o}$.
To verify local integrability,
let $K\subset\mathbb R^2\times[0,\infty)$ be compact and set
\begin{equation*}
    D_K
    :=\{(\mathbf k',\kappa):\kappa>0,\quad
      (\mathbf k',\sqrt{|\mathbf k'|^2+\kappa^2})\in K\}.
\end{equation*}
Changing variables in the integral gives
\begin{align*}
    \int_{K\cap\Lambda}|G_h(\mathbf k',\rho)|\,d\mathbf k'\,d\rho
    &=
    \int_{D_K}\frac{|\widehat h(\mathbf k',\kappa)|}
                     {\sqrt{|\mathbf k'|^2+\kappa^2}}\,d\mathbf k'\,d\kappa\\
    &\le
    \|\widehat h\|_{L^2(D_K)}
    \left(\int_{D_K}\frac{d\mathbf k'\,d\kappa}
                              {|\mathbf k'|^2+\kappa^2}\right)^{1/2}
    <\infty.
\end{align*}
Since $D_K$ is bounded and
$|\mathbf k|^{-1}\in L^2_{\mathrm{loc}}(\mathbb R^3)$, the density is
locally integrable, including at the vertex; oddness handles negative
$\rho$. Moreover, under the same change of variables,
$\rho=|\mathbf k|$ and $|\mathbf k'|\le\rho$. Hence the Euclidean size of
$(\mathbf k',\rho)$ is bounded above by $\sqrt{2}\,|\mathbf k|$.
Cauchy--Schwarz gives
\begin{equation*}
    \int_{\mathbb R^3}
    \frac{|\widehat h(\mathbf k)|}{|\mathbf k|}
    (1+|\mathbf k|)^{-2}\,d\mathbf k
    <\infty.
\end{equation*}
It follows that $G_h$ has at most polynomial growth and therefore defines
a tempered distribution. Equation~\eqref{eq:cone_density_pairing} gives
\begin{equation*}
    \langle\widehat{\gamma^o},\psi\rangle
    =\int_{\mathbb R^3}G_h(\mathbf k',\rho)\psi(\mathbf k',\rho)
      \,d\mathbf k'\,d\rho,
    \qquad\psi\in\mathcal S.
\end{equation*}
Therefore $\widehat{\gamma^o}$ is represented by $G_h$, which proves \eqref{eq:full_odd_cone_representation}.
\end{proof}


\section{Supporting estimates and constructions}
\label{app:supporting_estimates}

\subsection{A depth-cancelling family near tangential frequencies}
\label{app:near_tangential_family}

We construct a depth-cancelling irrotational Schwartz family for which
the standard Sobolev data norm cannot be bounded uniformly by the
reconstruction norm.

Choose nonzero real functions
$a\in C_c^\infty(\mathbb R^2)$ and
$b\in C_c^\infty((1,2))$, with $a$ even under
$\mathbf k'\mapsto-\mathbf k'$ and supported in
$1<|\mathbf k'|<2$. Regard $b$ as its zero extension to $\mathbb R$.
For $0<\varepsilon<1/2$, define
\begin{align*}
    \widehat h_\varepsilon(\mathbf k',k_3)
    &:=
    a(\mathbf k')
    \left[
        b(k_3/\varepsilon)+b(-k_3/\varepsilon)
    \right],\\
    \widehat{\mathbf u_\varepsilon}(\mathbf k)
    &:=
    \frac{i\mathbf k}{|\mathbf k|^2}
    \widehat h_\varepsilon(\mathbf k),
    \qquad \mathbf k\ne\mathbf0,
\end{align*}
with
$\widehat{\mathbf u_\varepsilon}(\mathbf0)=\mathbf0$.

The scalar spectrum is real and even in both $\mathbf k'$ and $k_3$.
Since its support is separated from the origin,
$\widehat{\mathbf u_\varepsilon}$ is smooth and compactly supported.
Hence
$\mathbf u_\varepsilon\in
H_e^s\cap\mathcal S(\mathbb R^3;\mathbb R^3)$ for every $s\ge0$, with
\[
    -\nabla\!\cdot\mathbf u_\varepsilon=h_\varepsilon,
    \qquad
    \nabla\times\mathbf u_\varepsilon=0.
\]
Moreover,
$\widehat h_\varepsilon(\mathbf k',0)=0$, so
$\bar h_\varepsilon=0$ by
\eqref{eq:depth_cancellation_condition}.

Let
$g_\varepsilon:=\mathcal M\mathbf u_\varepsilon$, and let
$\gamma_\varepsilon$ denote its preprocessed data.
Since $\mathbf u_\varepsilon$ is irrotational,
\cref{thm:canonical_reconstruction} gives
\[
    \mathcal Ig_\varepsilon=\mathbf u_\varepsilon.
\]

On the support of $\widehat h_\varepsilon$,
\[
    1\le|\mathbf k'|\le2,
    \qquad
    \varepsilon\le|k_3|\le2\varepsilon,
    \qquad
    1\le|\mathbf k|\le\sqrt5.
\]
For each fixed $s\ge0$, the ratio of the data weight in
\eqref{eq:data_spectral_norm} to the reconstruction weight in
\eqref{eq:field_spectral_norm} therefore satisfies
\begin{equation*}
    \frac{1}{2\varepsilon}
    \le
    \left(
        \frac{1+|\mathbf k'|^2+|\mathbf k|^2}
             {1+|\mathbf k|^2}
    \right)^s
    \frac{|\mathbf k|}{|k_3|}
    \le
    \frac{2^s\sqrt5}{\varepsilon}.
\end{equation*}
Integrating this comparison and using
\eqref{eq:field_spectral_norm} and
\eqref{eq:data_spectral_norm} gives
\begin{equation}
    \frac{1}{\sqrt{2\varepsilon}}
    \le
    \frac{\|\gamma_\varepsilon^o\|_{H^s}}
         {\|\mathcal Ig_\varepsilon\|_{H^s}}
    \le
    \frac{2^{s/2}5^{1/4}}{\sqrt{\varepsilon}}.
    \label{eq:near_tangential_ratio}
\end{equation}
In particular,
\[
    \frac{\|\gamma_\varepsilon^o\|_{H^s}}
         {\|\mathcal Ig_\varepsilon\|_{H^s}}
    \longrightarrow\infty
    \qquad\text{as }\varepsilon\downarrow0.
\]
Hence no constant $C_s$ can satisfy
\[
    \|\gamma^o\|_{H^s}
    \le
    C_s\|\mathcal Ig\|_{H^s}
\]
throughout the depth-cancelling Schwartz class.
Since the spectral support satisfies
$\varepsilon\le|k_3|\le2\varepsilon$ while
$1\le|\mathbf k|\le\sqrt5$, it approaches the tangential region
$k_3=0$ as $\varepsilon\downarrow0$.

\subsection{Spherical means at the radial origin}
\label{app:radial_origin_interface}

\begin{lemma}
\label{lem:radial_lift}
Let $h\in L^2(\mathbb R^3)$ be even in $x_3$, and define
\[
    G(\mathbf x',\mathbf z)
    :=
    \mathcal Rh(\mathbf x',|\mathbf z|),
    \qquad
    \mathbf z\in\mathbb R^3\setminus\{\mathbf0\}.
\]
Then, for every $R>0$,
\[
    G\in
    L^2\bigl(
        \mathbb R^2\times B_{\mathbb R^3}(0,R)
    \bigr),
\]
and $G$ agrees with the spherical-mean transform used in
\cite{Andersson1988SphericalAverages}.

Consequently, if $\mathcal Rh=0$ almost everywhere on
$V\times(0,R)$ for an open set $V\subset\mathbb R^2$, then this transform
vanishes on $V\times B_{\mathbb R^3}(0,R)$.
\end{lemma}

\begin{proof}
By \eqref{eq:preprocessed_slice_bound},
\begin{equation*}
\begin{aligned}
    \int_{\mathbb R^2}\int_{|\mathbf z|<R}
    |G(\mathbf x',\mathbf z)|^2
    \,d\mathbf z\,d\mathbf x'
    &=
    4\pi\int_0^R
    \|r\mathcal Rh(\cdot,r)\|_{L^2(\mathbb R^2)}^2\,dr\\
    &\le
    \pi R^2\|h\|_{L^2(\mathbb R^3)}^2.
\end{aligned}
\end{equation*}
Thus $G$ has the asserted local $L^2$ regularity.

Choose $x_3$-even functions
$h_j\in\mathcal S(\mathbb R^3)$ with
$h_j\to h$ in $L^2(\mathbb R^3)$, and let $G_j$ denote their radial
lifts. The same estimate applied to $h_j-h$ gives
\[
    G_j\to G
    \quad\text{in }
    L^2\bigl(
        \mathbb R^2\times B_{\mathbb R^3}(0,R)
    \bigr)
\]
for every $R>0$.

For Schwartz functions, $G_j$ agrees with the spherical-mean transform
in \cite{Andersson1988SphericalAverages}. Since $h_j\to h$ in
$\mathcal S'$ and this transform is continuous
\cite{Andersson1988SphericalAverages},
the transforms of $h_j$ converge to the transform of $h$.
Since $G_j\to G$ locally in $L^2$, the latter transform is represented
by $G$.

The final assertion follows from the definition of $G$.
\end{proof}

\section*{Acknowledgments}
  The authors thank Xintong Liu and Siqin Zheng for fruitful discussions.
  The authors used ChatGPT to assist with manuscript 
  revision, including mathematical exposition. The authors assume responsibility for all
  content.

\bibliographystyle{siamplain}
\bibliography{references}

\end{document}